\documentclass{amsart}
\usepackage{latexsym,bbm,amsxtra,amscd,ifthen}
\usepackage{amsfonts}
\usepackage{verbatim}
\usepackage{amsmath}
\usepackage{amsthm}
\usepackage{amssymb}
\usepackage{url}
\usepackage[all]{xy}
\usepackage[colorlinks=true,citecolor=blue,linkcolor=blue]{hyperref}

\theoremstyle{plain}

\newtheorem{theorem}{Theorem}
\newtheorem{lemma}[theorem]{Lemma}
\newtheorem{proposition}[theorem]{Proposition}
\newtheorem{corollary}[theorem]{Corollary}

\numberwithin{theorem}{section}
\numberwithin{equation}{theorem}

\theoremstyle{definition}
\newtheorem{definition}[theorem]{Definition}
\newtheorem{setup}[theorem]{Setup}

\newtheorem{remark}[theorem]{Remark}

\newtheorem*{question*}{Question}

\DeclareMathOperator{\PHom}{PHom}
\DeclareMathOperator{\SHom}{\underline{SHom}}
\DeclareMathOperator{\uPHom}{\underline{PHom}}

\DeclareMathOperator{\End}{End}
\DeclareMathOperator{\Ext}{Ext}
\DeclareMathOperator{\Tor}{Tor}
\DeclareMathOperator{\Hom}{Hom}

\DeclareMathOperator{\injdim}{injdim}
\DeclareMathOperator{\pdim}{projdim}

\DeclareMathOperator{\gldim}{gldim}
\DeclareMathOperator{\rgldim}{r.gldim}

\DeclareMathOperator{\im}{im}
\DeclareMathOperator{\depth}{depth}

\DeclareMathOperator{\gr}{gr}

\DeclareMathOperator{\mcm}{mcm}

\DeclareMathOperator{\uExt}{{\underline{Ext}}}
\DeclareMathOperator{\uHom}{{\underline{Hom}}}
\DeclareMathOperator{\uEnd}{{\underline{End}}}
\DeclareMathOperator{\coker}{coker}
\DeclareMathOperator{\add}{add}

\def\kk{\mathbbm{k}}
\def\mmod{\operatorname{mod}}

\def\gr{\operatorname{gr}}
\def\tor{\operatorname{tor}}
\def\Gr{\operatorname{Gr}}
\def\QGr{\operatorname{QGr}}
\def\qgr{\operatorname{qgr}}

\begin{document}

\title[Quasi-resolutions]
{Pre-resolutions of noncommutative isolated singularities}

\author{Ji-Wei He and Yu Ye}

\address{He: Department of Mathematics,
Hangzhou Normal University,
Hangzhou Zhejiang 311121, China}
\email{jwhe@hznu.edu.cn}
\address{Ye: School of Mathematical Sciences, University of Science and Technology of China, Hefei Anhui 230026, China\hfill \break
\indent\qquad CAS Wu Wen-Tsun Key Laboratory of Mathematics, University of Science and Technology of China, Hefei, Anhui, 230026, PR China}
\email{yeyu@ustc.edu.cn}

\begin{abstract} We introduce the notion of right pre-resolutions (quasi-resolutions) for noncommutative isolated singularities, which is a weaker version of quasi-resolutions introduced by Qin-Wang-Zhang in \cite{QWZ}. We prove that right quasi-resolutions for noetherian bounded below and locally finite graded algebra with right injective dimension 2 are always Morita equivalent. When we restrict to noncommutative quadric hypersurfaces, we prove that a noncommutative quadric hypersurface, which is a noncommutative isolated singularity, always admits a right pre-resolution. Besides, we provide a method to verify whether a noncommutative quadric hypersurface is an isolated singularity. An example of noncommutative quadric hypersurfaces with detailed computations of indecomposable maximal Cohen-Macaulay modules and right pre-resolutions is included as well.
\end{abstract}

\subjclass[2010]{16S37, 16E65, 16G50}

%16B50 (1991-now) Category-theoretic methods and results
%16D90 (1991-now) Module categories [See also 16Gxx, 16S90];
%module theory in a category-theoretic context; Morita equivalence and duality
%16E10 (1991-now) Homological dimension
%16E35 (2010-now) Derived categories
%16E65 (2000-now) Homological conditions on rings (generalizations of
%regular, Gorenstein, Cohen-Macaulay rings, etc.)

\keywords{Right pre-resolution, noncommutative isolated singularity, noncommutative quadric hypersurface}

%\thanks{ }

\maketitle

%\tableofcontents

\setcounter{section}{-1}
\section{Introduction}

Let $R$ be a commutative normal Gorenstein domain. Van den Bergh introduced the notion of noncommutative crepant resolutions of $R$ in \cite{vdB}. Roughly speaking, a noncommutative crepant resolution of $R$ is an $R$-algebra of the form $\Lambda=\End_R(M)$, where $M$ is a reflexive $R$-module. Iyama-Reiten extended the notion of noncommutative crepant resolutions to module-finite commutative algebras over a noetherian commutative Cohen-Macaulay ring (cf. \cite{IR}). Let $R$ be a commutative Cohen-Macaulay equi-codimensional normal Gorenstein domain with a canonical module. It has been proven that noncommutative crepant resolutions of $R$ are always derived equivalent provided $\dim R\leq 3$ (cf. \cite[Corollary 8.8]{IR} and \cite[Theorem 1.5]{IW}).

In order to study noncommutative singularities, Qin-Wang-Zhang extended the notion of noncommutative resolutions to noncommutative algebras which are possibly not module-finite over their centers (cf. \cite{QWZ}). Let $A$ be a (both left and right) noetherian algebra over a field $\kk$, and let $\partial$ be a symmetric dimension function of the category of right $A$-modules. Two right $A$-modules $M$ and $N$ are said to be $s$-isomorphic (cf. \cite{QWZ}) if there is a right $A$-module $P$ and two homomorphisms $f\colon P\to M$ and $g\colon P\to N$ such that the $\partial$-dimensions of the kernels and the cokernels of $f$ and $g$ are no larger than $s$. The following definition was given in \cite{QWZ}.

\begin{definition}\label{def0} Let $A$ be a (both left and right) noetherian algebra with $\partial$-dimension $d$. If there is a noetherian Auslander regular $\partial$-Cohen-Macaulay algebra $B$ with $\partial$-dimension $d$ and two finitely generated bimodules ${}_BM_A$ and ${}_AN_B$ such that $M\otimes_AN$ is $(d-2)$-isomorphic to $B$ and $N\otimes_BM$ is $(d-2)$-isomorphic to $A$ as bimodules, then $B$ is called a {\it noncommutative quasi-resolution} of $A$.
\end{definition}

Qin-Wang-Zhang proved that noncommutative quasi-resolutions of a noetherian algebra $A$ with $\partial$-dimension 2 are Morita equivalent. If $\partial$-dimension of $A$ is 3, then noncommutative quasi-resolutions of $A$ are derived equivalent (with further assumptions on $A$, cf. \cite[Theorem 0.6]{QWZ}). Thus, they generalized the corresponding results in \cite{IR} and \cite{IW}.

If further, $A$ is Auslander Gorenstein and $\partial$-Cohen-Macaulay, then the algebra $B$ in Definition \ref{def0} is isomorphic to the endomorphism algebra $\End_A(U)$ for some bimodule ${}_BU_A$ which is reflexive on both sides (cf. \cite[Corollary 3.13]{QWZ}).

Unlike the commutative case, given a noncommutative noetherian algebra $A$ and a finitely generated right $A$-module $U$, it is usually a tough task to check whether $\End_A(U)$ is a noetherian algebra. In this sense, to find a noncommutative quasi-resolution of a noetherian algebra is not an easy job in general.

In this paper, we only consider the noncommutative resolutions of noncommutative graded isolated singularities (cf. Section \ref{sectnr}), which allows us to drop some restrictions on the algebras as given in Definition \ref{def0}.

Now let $A$ be a bounded below graded algebra, that is, $A=\oplus_{i\in\mathbb Z}A_i$ with $A_i=0$ for $i<<0$. Assume $A$ is right noetherian and locally finite. Let $\gr A$ be the category of finitely generated right graded $A$-modules, and $\tor A$ the subcategory of $\gr A$ consisting of finite dimensional modules. Let $\qgr A=\gr A/\tor A$. We introduce the following weaker version of noncommutative resolutions of right noetherian algebras, which is much closer to Van den Bergh's original definition of noncommutative crepant resolution (cf. \cite{vdB}).

\begin{definition}(More precisely, see Definition \ref{def3}) Let $A$ be a right noetherian graded algebra which is bounded below and locally finite with injective dimension $\injdim A_A=d<\infty$. If there is a small maximal Cohen-Macaulay (cf. Section \ref{sectnr}) module $M_A$ with $B=\End_A(M)$ such that

{\rm(i)} right global dimension $\rgldim (B)=d$,

{\rm(ii)} the functor $\uHom_A(M,-)\colon \gr A\longrightarrow\gr B$ induces an equivalence $\qgr A\cong\qgr B$, \\
then we call $B$ a {\it right pre-resolution} of $A$.

If further, $B$ is {\it right generalized Artin-Schelter regular} (cf. Definition \ref{def2}), then we call $B$ a {\it right quasi-resolution} of $A$.
\end{definition}

Then we have the following result (cf. Theorem \ref{thm3}) parallel to the ones in \cite{IR,IW,QWZ}. We remark that our proof is quite different from theirs. In fact, the algebras considered in \cite{IR,IW,QWZ} are assumed to be noetherian on both sides, while we only assume the right noetherianness here.

\begin{theorem} Let $A$ be a right noetherian graded algebra which is bounded below and locally finite with injective dimension $\injdim A_A=2$. If $A$ has a right quasi-resolution, then

{\rm(i)} $A$ is CM-finite, that is, there are only finitely many nonisomorphic indecomposable maximal Cohen-Macaulay (MCM) right $A$-modules (up to degree shifts);

{\rm(ii)} any two right quasi-resolutions of $A$ are graded Morita equivalent.
\end{theorem}

If $A$ is an AS-Gorenstein algebra (cf. Definition \ref{def1}) which is a noncommutative isolated singularity, then the CM-finiteness will induce the existence of right pre-resolutions, as shown in the following results (cf. Theorems \ref{thm1} and \ref{thm2}).

\begin{theorem} Let $A$ be an AS-Gorenstein algebra which is a noncommutative isolated singularity.
\begin{itemize}
  \item [(i)] Let $M_A$ be an MCM module. Then $\uEnd_A(M)$ is a right noetherian graded algebra.
  \item [(ii)] Assume that $A$ is CM-finite and $\injdim A\ge2$. Let $\{P_0=A,P_1,\dots,P_n\}$ be the set of all the nonisomorphic indecomposable MCM modules (up to degree shifts). Let $M=\bigoplus_{i=1}^nP_n\oplus A$. Then $B:=\uEnd_A(M)$ is a right pre-resolution of $A$.
\end{itemize}
\end{theorem}

Let $S$ be a quantum polynomial algebra (cf. Section \ref{secthyper}), that is, $S$ is a Koszul AS-regular algebra with Hilbert series $H_S(t)=\frac{1}{(1-t)^n}$ for some $n\ge1$. Pick a central regular element $\varpi\in A$ of degree 2. The quotient algebra $A=S/S\varpi$ is called a noncommutative quadric hypersurface. Note that a noncommutative quadric hypersurface is CM-finite if and only if it is a noncommutative isolated singularity (cf. \cite[Theorem 4.13]{MU}). Hence a noncommutative quadric hypersurface $A$ which is also a noncommutative isolated singularity always admits a right pre-resolution, and to find such a pre-resolution it suffices to compute all the indecomposable MCM-modules of $A$.

Assume $A$ is a noncommutative quadric hypersurface with injective dimension $d$. Let $\Omega^d(\kk_A)$ be the $d$-th syzygy of the trivial module $\kk_A$. Set $$\mathbb M:=\Omega^d(\kk_A)(d).$$ Then $\mathbb M$ is a Koszul MCM module. Associated to $A$, Smith-Van den Bergh constructed a finite dimensional algebra $C(A)$, and proved that the stable category of MCM modules over $A$ is equivalent to the derived category of $C(A)$. In this paper, we prove the following results, which provides a relatively easy way to compute the algebra $C(A)$ and to find all the indecomposable MCM modules of $A$. Consequently, we show a method to construct a right pre-resolution of $A$ in case $A$ is a CM-finite.

\begin{theorem}\label{thm0quadric}(cf. Theorems \ref{qthm1} and \ref{thm-res-iso}) Let $S$ be a quantum polynomial algebra, and let $\varpi\in A$ be a central regular element of degree 2. Set $A:= S/S\varpi$. We have the follow consequences.
\begin{itemize}
  \item [(i)] $\End_{\gr A}(\mathbb M)\cong C(A)$;
  \item [(ii)] $A$ is a noncommutative isolated singularity if and only if $\End_{\gr A}(\mathbb M)$ is semisimple.
  \item [(iii)] Assume that $A$ is a noncommutative isolated singularity.
Then $B=\uEnd_A(\mathbb M\oplus A)$ is a right pre-resolution of $A$.
\end{itemize}
\end{theorem}
We remark that the second statement of the above theorem follows from \cite[Theorem 6.3]{HY} (see also \cite[Theorem 4.13]{MU}).

In the view of Theorem \ref{thm0quadric}, some properties of indecomposable MCM modules of $A$ are obtained in Section \ref{secthyper}. We provide a concrete example of quadric hypersurfaces with detailed computations of indecomposable MCM modules and right pre-resolutions in the last section.

\section{Preliminaries}
Throughout, $\kk$ will be a field of characteristic zero and all algebras considered are over $\kk$.

Let $A$ be a $\mathbb Z$-graded $\kk$-algebra. We denote by $\Gr A$ the category of right graded $A$-modules, and by $\Hom_{\Gr A}(M,N)$ the set of homogeneous right $A$-module homomorphisms which preserve the degrees of elements for $M, N\in \Gr A$. For $k\in \mathbb Z$, $M(k)$ is the right graded $A$-module such that $M(k)_n=M_{n+k}$. We write $\uHom_A(M,N)=\bigoplus_{k\in\mathbb Z}\Hom_{\Gr A}(M,N(k))$, and $\uExt_A^i(M,N)$ for the derived functor of $\uHom_A(M,N)$. In particular, we write $\uEnd_A(M)$ for $\uHom_A(M,M)$.

Let $\PHom_{\Gr A}(M,N)$ be the subset of $\Hom_{\Gr A}(M,N)$ consisting of homomorphisms which factor through some graded projective modules, and let $\uPHom_A(M,N)=\bigoplus_{k\in\mathbb Z}\PHom_{\Gr A}(M,N(k))$. The set of graded stable homomorphism is denoted by $\SHom_A(M,N)=\uHom_A(M,N)/\uPHom_A(M,N)$.

Let $A$ be a right noetherian graded algebra, and let $\gr A$ be the full subcategory of $\Gr A$ consisting of finitely generated modules. A right graded module $M\in \Gr A$ is called a {\it torsion} module if for every $m\in M$, the right submodule $mA$ is finite dimensional. Let $\Tor A$ be the full subcategory of $\Gr A$ consisting of all the torsion modules, and let $\tor A$ be the full subcategory of $\Tor A$ consisting of finite dimensional ones. Since $A$ is right noetherian, $\Tor A$ ( resp. $\tor A$) is a Serre subcategory of $\Gr A$ (resp. $\gr A$). The quotient categories $\QGr A=\Gr A/\Tor A$ and $\qgr A=\gr A/\tor A$ are both abelian and $\qgr A$ is an abelian subcategory of $\QGr A$. Let $\pi\colon\Gr A\to \QGr A$ be the projection functor. Then $\pi$ has a right adjoint functor $\omega\colon \QGr A\to \Gr A$ such that $\pi\omega\cong id$.

For $M\in \Gr A$, we write $\mathcal{M}=\pi(M)$. The Hom-sets in $\QGr A$ is defined by $$\Hom_{\QGr A}(\mathcal{M},\mathcal{N})=\underset{\longrightarrow}\lim\Hom_{\Gr A}(M',N/K),$$ where the limit runs over all the pairs $(M',K)$ with $K\subseteq N$ and $M'\subseteq M$ such that both $K$ and $M/M'$ are torsion modules. We refer to \cite{AZ} for more information about the quotient categories.

A graded algebra $A$ is {locally finite} if $\dim A_n<\infty$ for all $n\in \mathbb Z$, and $A$ is {\it bounded below} if $A_n=0$ for all $n<<0$. If $A_0=\kk$ and $A_n=0$ for all $n<0$, then $A$ is said to be {\it connected} graded.

We recall the following classical definition (cf. \cite{AS}).
\begin{definition}\label{def1} Let $A$ be a (both left and right) noetherian connected graded algebra. $A$ is called an {\it Artin-Schelter Gorenstein} algebra of dimension $d$ if
\begin{enumerate}
  \item [(i)] $\injdim{}_AA=\injdim A_A=d<\infty$;
  \item [(ii)] $\uExt_A^i(\kk,A)=0$ for all $i\neq d$, and $\uExt^d_A(\kk,A)=\kk(l)$ for some $l$;
  \item [(iii)] the left version of (ii) is satisfied.
\end{enumerate}

\noindent If further, $\gldim A=d$, then $A$ is called an {\it Artin-Schelter regular} algebra. The integer $l$ is usually called
the {\it Gorenstein parameter} of $A$.
\end{definition}

We need a more general version of Artin-Schelter Gorenstein algebras in this paper. Without otherwise statement, we always assume that $A$ is a right noetherian graded algebra which is locally finite and bounded below. Let $J(A)$ be the graded Jacobson radical of $A$. We recall the following well-known facts without giving a proof.

\begin{lemma}\label{lem0} Retain the notations as above. Then the following statements hold true.
\begin{enumerate}
  \item[(i)] $A/J(A)$ is finite dimensional.
  \item[(ii)] Let $J(A_0)$ be the Jacobson radical of $A_0$. Then $J(A_0)=J(A)\cap A_0$.
  \item[(iii)] There is an integer $n_0$ such that $J(A)\supseteq A_{\ge n_0}$, and $\bigcap_{n\ge0}J(A)^n=0$.
\end{enumerate}
\end{lemma}

We generalize Artin-Schelter Gorenstein algebras to right noetherian bounded below algebras.

\begin{definition}\label{def2} Let $A$ be as above and $J=J(A)$ be the graded Jacobson radical of $A$. We call $A$ a {\it right generalized Artin-Schelter Gorenstein} algebra of dimension $d$ if

(i) $\injdim A_A=d<\infty$,

(ii) $\uExt^i_A(A/J,A)=0$ for all $i\neq 0$,

(iii) as a left $A$-module $\uExt_A^d(A/J,A)$ is annihilated by $J$, and $\uExt^d_A(A/J,A)$ is invertible as a graded $A/J$-$A/J$-bimodule.

If further, the right global dimension $\rgldim (A)=d$, then we call $A$ a {\it right generalized Artin-Schelter regular} algebra.
\end{definition}

Below, the phase ``Artin-Schelter'' is simply denoted by ``AS'', and ``generalized Artin-Schelter'' by ``GAS''.

\begin{remark}{\rm } The notion of GAS-Gorentein algebras is a slight generalization of the ones in \cite[Definition 1.4]{RR} and \cite[Definition 3.1]{MM}, where the algebras considered are $\mathbb N$-graded.
\end{remark}

Let $\Gamma\colon\gr A\to \gr A$ be the torsion functor, that is, $\Gamma(M)=\{m\in M\mid\dim_\kk(mA)<\infty\}$. The $i$-th right derived functor of $\Gamma$ is denoted by $R^i\Gamma$. By Lemma \ref{lem0}, we have $\Gamma\cong \underset{n\to\infty}\lim\uHom_A(A/J^n,-)$. For $M\in \gr A$, the {\it depth} of $M$ is defined to be
$$\depth(M)=\min\{i\mid R^i\Gamma(M)\neq0\}.$$ Then $\depth(M)$ is either a non-negative integer or $\infty$.
The following lemma is classical for connected graded algebras, and the proof for connected case applies also to our case.

\begin{lemma}\label{lemdf} $\depth(M_A)=\min\{i|\uExt_A^i(A/J,M)\neq0\}$.
\end{lemma}

Recall that $\pi\colon \Gr A\to \QGr A$ has a right adjoint functor $\omega\colon \QGr A\to \Gr A$. The following lemma is clear.

\begin{lemma}\label{lemss} Let $M_A$ be a finitely generated module. If $\depth(M_A)\ge2$, then $\omega\pi(M)\cong M$.
\end{lemma}

A useful homological identity in the theory of AS-Gorenstein algebras is Auslander-Buchsbaum formula (cf. \cite[Theorem 3.2]{Jo} for a noncommutative version), which provides an effective way to calculate the depth of a module over a local ring. For our purpose, it will be handful to have a more general version of the Auslander-Buchsbaum formula for right GAS-Gorenstein algeras. We mention that the proof is a modification of that of \cite[Theorem 3.2]{Jo}.

\begin{theorem}[Auslander-Buchsbaum formula]\label{thm0} Let $A$ be a right GAS-Gorenstein algebra, and  $M_A\in\gr A$. Suppose that the projective dimension $\pdim(M_A)<\infty$. Then  $$\pdim(M_A)+\depth(M_A)=\depth(A_A).$$
\end{theorem}
\begin{proof} If $\injdim A_A=0$, then it is clear. Assume $\injdim(A_A)=d\ge1$ and $\pdim(M_A)=p$. Take a graded projective resolution of the right module $A/J$:
$$\cdots\to P^{-2}\to P^{-1}\to P^0\to A/J\to0,$$  where each $P^i$ is finitely generated for all $i$. Applying $\uHom_A(-,A)$ to the resolution, we obtain the following sequence:
\begin{equation}\label{seq1}
  0\longrightarrow\uHom_A(P^0,A)\longrightarrow\uHom_A(P^{-1},A)\longrightarrow\uHom_A(P^{-2},A)\longrightarrow\cdots.
\end{equation}
Take a minimal graded projective resolution
\begin{equation}\label{seq2}
  0\to Q^{-p}\to\cdots\to Q^{-1}\to Q^0\to M\to0
\end{equation}
of $M$. By taking the tensor product of (\ref{seq2}) and (\ref{seq1}) we obtain a double complex
{\small$$\xymatrix{
\vdots\ar[d]&\vdots\ar[d]&\vdots\ar[d]& \\
Q^{-2}\otimes_A\uHom_A(P^0,A)\ar[d] \ar[r]&Q^{-2}\otimes_A\uHom_A(P^{-1},A)\ar[d] \ar[r]&Q^{-2}\otimes_A\uHom_A(P^{-2},A)\ar[d] \ar[r]&\cdots\\
  Q^{-1}\otimes_A\uHom_A(P^0,A)\ar[d] \ar[r] & Q^{-1}\otimes_A\uHom_A(P^{-1},A)\ar[d] \ar[r]&Q^{-1}\otimes_A\uHom_A(P^{-2},A)\ar[d] \ar[r]&\cdots \\
  Q^0\otimes_A\uHom_A(P^0,A) \ar[r] & Q^0\otimes\uHom_A(P^{-1},A)\ar[r]& Q^0\otimes\uHom_A(P^{-2},A)\ar[r]&\cdots,  }$$}
which implies a convergent spectral sequence
\begin{equation}\label{spec1}
  E^{rs}_2=\Tor^A_r(M,\uExt^s_A(A/J,A))\Longrightarrow \uExt_A^{s-r}(A/J,M).
\end{equation}

By assumption, $\uExt^s_A(A/J,A)=0$ for $s\neq d$ and $\uExt^d_A(A/J,A)\cong V$ for some graded invertible $A/J$-$A/J$-bimodule $V$. Thus the spectral sequence (\ref{spec1}) collapses at the second level. Since the resolution (\ref{seq2}) is minimal, we have $\Tor^A_s(M,V)\cong Q^{-s}\otimes_A V$, and hence $\Tor^A_s(M,V)\neq 0$ for each $0\leq s\leq p$ for $V$ is invertible. Therefore, $\uExt_A^{d-p}(A/J,M)\neq0$ and $\uExt_A^i(A/J,M)=0$ for all $i<d-p$. By Lemma \ref{lemdf}, $\depth(M)=d-p$. The Auslander-Buchsbaum formula follows.
\end{proof}

\section{Noncommutative resolutions}\label{sectnr}
Noncommutative crepant resolutions for commutative Gorenstein algebras were introduced by Van den Bergh (cf. \cite{vdB}). Later, Qin-Wang-Zhang introduced noncommutative quasi-resolutions for noncommutative Auslander Gorenstein algebra in \cite{QWZ}.  In this section, we will modify the definition of noncommutative resolutions of \cite{QWZ} and give a weaker version of noncommutative resolutions for noncommutative isolated singularities.

Let $A$ be a right noetherian graded algebra which is bounded below and locally finite. Recall that $A$ is called a {\it noncommutative isolated singularity} (cf. \cite{Ue}),
if $\qgr A$ has finite global dimension, i.e., there is an integer $n_0\ge0$ such that $\Ext_{\qgr A}^i(\mathcal{M},\mathcal{N})=0$ for all $i>n_0$ and $M,N\in \gr A$.

Recall that a finitely generated right graded $A$-module $M$ is said to be {\it small} if the following conditions are satisfied:

(i) $\uEnd_A(M)$ is a right noetherian graded algebra, and

(ii) $\uHom_A(M,N)$ is a finitely generated right graded $\uEnd_A(M)$-module for any $N\in\gr(A)$ .

Let $M_A$ be small and let $B=\uEnd_A(M)$. Since $M$ is finitely generated and $A$ is bounded below and locally finite, $B$ is also bounded below and locally finite, and we have an additive functor $$F=\uHom_A(M,-)\colon\gr A\longrightarrow\gr B.$$

If $K$ is a finite dimensional right graded $A$-module, then $\uHom_A(M,K)$ is a finite dimensional right graded $B$-module. For $X,Y\in \gr A$, let $f\colon X\to Y$ be a homomorphism such that both $\ker f$ and $\coker f$ are finite dimensional. Let $f_*=\uHom_A(M,f)$. It is not hard to see both $\ker f_*$ and $\coker f_*$ are finite dimensional. Therefore, the functor $F$ induces a functor $\mathcal{F}\colon\qgr A\longrightarrow\qgr B$ which fits into the following commutative diagram
\begin{equation}\label{diag1}
  \xymatrix{
  \gr A \ar[d]_{\pi} \ar[r]^{F} & \gr B \ar[d]^{\pi} \\
  \qgr A \ar[r]^{\mathcal{F}} & \qgr B.   }
\end{equation}

Suppose that $A$ has finite injective dimension $\injdim A_A=d$. Recall that a finitely generated right graded $A$-module $M$ is called a {\it maximal Cohen-Macaulay module} (MCM module, for simplicity) if $R^i\Gamma(M)=0$ for all $i\neq d$.

\begin{definition}\label{def3} {\rm Let $A$ be a right noetherian graded algebra which is bounded below and locally finite with injective dimension $\injdim A_A=d<\infty$. If there is a small MCM module $M_A$ such that

{\rm(i)} right global dimension $\rgldim (B)=d$,

{\rm(ii)} the functor $\mathcal{F}$, as in the diagram (\ref{diag1}), is an equivalence,

\noindent then we call $B$ a {\it right pre-resolution} of $A$.

If further, $B$ is right GAS-regular, then we call $B$ a {\it right quasi-resolution} of $A$.
}
\end{definition}

\begin{remark}(1) A right noetherian graded algebra which admits a right pre-resolution is automatically a noncommutative isolated singularity. This follows from the well-known fact that the global dimension of $\qgr B$ is no greater than the one of $\gr B$, see for instance \cite[Section 7]{AZ}.

(2) The above definition is a modification of \cite[Definition 0.5]{QWZ} for noncommutative isolated singularities, where the algebra $B$ is assumed to be a (both left and right) noetherian Auslander regular $\mathbb N$-graded algebra. We will show some examples of right pre-resolutions and right quasi-resolutions for right GAS-Gorenstein algebras in the subsequent sections. We consider non-positively graded algebras because a noncommutative resolution of a noetherian algebra may not positively graded in general.
\end{remark}

A right noetherian graded locally finite algebra with finite right injective dimension is said to be {\it CM-finite} if it has, up to degree shifts, finitely many nonisomorphic indecomposable MCM modules.

Our main result of this section is as follows. It may be viewed as a noncommutative version of \cite[Theorem 1.5]{IW} in dimension 2 case.

\begin{theorem} \label{thm3} Let $A$ be a right noetherian graded algebra which is bounded below and locally finite with injective dimension $\injdim A_A=2$. Assume that $A$ has a right quasi-resolution. Then

{\rm(i)} $A$ is CM-finite;

{\rm(ii)} any two right quasi-resolutions of $A$ are graded Morita equivalent.
\end{theorem}
\begin{proof} (i) By assumption, there exists some small MCM module $M_A$ such that $B=\uEnd_A(M)$ is a right quasi-resolution of $A$. Consider the functor $F=\uHom_A(M,-)\colon \gr A\to \gr B$ and the induced functor $\mathcal F\colon \qgr A\to \qgr B$. Then we have commutative diagram (\ref{diag1}).
Take an MCM $A$-module $N$, and let $X=\uHom_A(M,N)$. Then $X\in\gr B$. Since $\mathcal{F}$ is an equivalence, we have
$$\Hom_{\qgr A}(\pi(M),\pi(N)(k))\cong \Hom_{\qgr B}(\mathcal{F}(\pi(M)),\mathcal{F}(\pi(N)(k))),\text{ for all } k\in\mathbb Z.$$
By definition
\begin{eqnarray*}
   \Hom_{\qgr B}(\mathcal{F}(\pi(M)),\mathcal{F}(\pi(N)(k)))  &=& \Hom_{\qgr B}(\pi(\uHom_A(M,M)),\pi(\uHom_A(M,N))(k)) \\
     &=& \Hom_{\qgr B}(\pi(B),\pi(X)(k)).
  \end{eqnarray*}

Let $J$ be the graded Jacobson radical of $B$. Since $B/J$ is finite dimensional and $X$ is finitely generated, we have
$$\Hom_{\qgr B}(\pi(B),\pi(X)(k))=\underset{n\to\infty}\lim\Hom_{\gr B}(J^n,X(k)).$$
From the exact sequence $0\to J^n \to B\to B/J^n\to 0$ we obtain the following exact sequence
\begin{eqnarray}\label{seq3}
\nonumber 0\longrightarrow\underset{n\to\infty}\lim\Hom_{\gr B}(B/J^n,X(k))\longrightarrow&\Hom_{\gr B}(B,X(k))\longrightarrow\underset{n\to\infty}\lim\Hom_{\gr B}(J^n,X(k))\\
&\longrightarrow\underset{n\to\infty}\lim\Ext^1_{\gr B}(B/J^n,X(k))\longrightarrow0.
\end{eqnarray}
By the commutative diagram (\ref{diag1}), we have the following commutative diagram
\begin{equation}\label{diag2}
  \xymatrix{
    \Hom_{\gr B}(B,X(k))\ar[r]&\underset{n\to\infty}\lim\Hom_{\gr B}(J^n,X(k)) \\
    \Hom_{\gr B}(B,X(k))\ar[r]\ar[u]_{=}&\Hom_{\qgr B}(\pi(B),\pi(X)(k))\ar[u]_{=}\\
    \Hom_{\gr A}(M,N(k)) \ar[r]\ar[u]_{F} & \Hom_{\qgr A}(\pi(M),\pi(N)(k))\ar[u]_{\mathcal{F}}.   }
\end{equation}
Note that $X_k=\uHom_A(M,N)(k)=\Hom_{\gr A}(M,N(k))$. It follows that the maps in the left column of the diagram (\ref{diag2}) are bijective. Since $\mathcal{F}$ is an equivalence, the maps in the right column are bijective. Since $N$ is MCM, then $\depth(N)=2$. It follows from Lemma \ref{lemss} that $\Hom_{\qgr A}(\pi(M),\pi(N)(k))\cong\Hom_{\gr A}(M,N(k))$. Therefore the bottom map in the diagram (\ref{diag2}) is also bijective. Hence the top map in the diagram (\ref{diag2}) is an isomorphism.

By the exact sequence (\ref{seq3}), we have, for all $k\in\mathbb Z$, $\underset{n\to\infty}\lim\Hom_{\gr B}(B/J^n,X(k))=0$ and $\underset{n\to\infty}\lim\Ext^1_{\gr B}(B/J^n,X(k))=0$. Hence $\Gamma(X)=R^1\Gamma(X)=0$. It follows that $\depth(X)\ge2$. Since  $\rgldim(B)=\injdim A_A=2$, $X$ is a projective $B$-module by Auslander-Buchsbaum formula (cf. Lemma \ref{lem0}). Hence $X\in \add(B)$, where $\add(B)$ is the category of direct summands of direct sums of finite copies of degree shifts of $B$. Therefore $\pi(X)\in \add(\pi(B))$ and hence $\pi(N)\in \add(\pi(M))$ for $\mathcal{F}$ is an equivalence. Since both $M$ and $N$ are MCM, it follows that $\omega\pi(M)\cong M$ and $\omega\pi(N)\cong N$ (cf. Lemma \ref{lemss}). Hence $N\in \add(M)$, which implies that every indecomposable MCM module is a direct summand of $M$ (up to a degree shift). Hence $A$ is CM-finite.

(ii) Suppose that $N'$ is another small MCM module such that $B'=\uEnd_A(M)$ is a right quasi-resolution of $A$. By the proof of (i), $N'\in\add(M)$ and $M\in \add(N')$. Therefore $\add(M)=\add(N')$. Hence $B$ and $B'$ are graded Morita equivalent.
\end{proof}
\begin{remark} In \cite{QWZ}, the authors extended the theory of noncommutative crepant resolutions for commutative algebras to noncommutative settings. They proved that noncommutative quasi-resolutions for a noetherian $\mathbb N$-graded algebra with Gelfand-Kirillov dimension 2 are always Morita equivalent (cf. \cite[Theorem 0.6(1)]{QWZ}), which extensively generalizes a similar result in commutative case (cf. \cite[Theorem 1.5]{IW}). We remark that a noncommutative quasi-resolution $B$ in \cite{QWZ} is assumed to be (left and right) noetherian $\mathbb{N}$-graded Auslander regular and Cohen-Macaulay. In contrast, we assume that the resolution $B$ is right GAS-regular. Moreover, the method we used are different from that of \cite{QWZ}.
\end{remark}

\section{Endomorphism rings of CM-finite AS-Gorenstein algebras}
In this section, $A$ is an AS-Goresntein algebra. We will show that the endomorphism ring of an MCM module over a noncommutative isolated singularity is always right noetherian, which suggests the existence of resolutions for noncommutative isolated singularities.

\begin{lemma}\label{lem1} Let $A$ be an AS-Gorenstein algebra which is a noncommutative isolated singularity. Let $M_A$ be an MCM module, and let $N_A$ be a finitely generated graded $A$-module. Then $\SHom_A(M,N)$ is finite dimensional.
\end{lemma}
\begin{proof} Assume that the injective dimension $\injdim(A_A)=d$. Since $M$ is an MCM module, there is an exact sequence $$0\to M\overset{\tau}\to P^{-d}\overset{\partial^{-d}}\to P^{-d+1}\to\cdots\to P^0\overset{\partial^0}\to P^1\to\cdots,$$ where $P^i$ is a finitely generated graded projective $A$-module for all $i\ge -d$. Let $X=\im\partial^0$. Then $\uExt^{d+1}_A(X,N)=\uHom_A(M,N)/\im(\tau^*)$, where $\tau^*\colon \uHom_A(P^{-d},N)\to\uHom_A(M,N)$ is the induced map. Moreover $\im(\tau^*)=\uPHom_A(M,N)$ since $M$ is an MCM module. Hence $\uExt^{d+1}_A(X,N)=\uHom_A(M,N)/\im(\tau^*)=\SHom_A(M,N)$. By assumption $A$ is a noncommutative isolated singularity, it follows that $\uExt^{d+1}_A(X,N)$ is finite dimensional (cf. \cite[Proposition 4.3]{SvdB}).
\end{proof}

We fix a notation which will be used in the proof of the next theorem. For an integer $i$ and a graded $A$-module $X$, we define a graded homomorphism $s^i\colon X(i)\to X$ by setting $s^i(x)=x$ for all $x\in X$. Then $s^i$ is a graded homomorphism of degree $i$.

\begin{theorem}\label{thm1} Let $A$ be an AS-Gorenstein algebra which is a noncommutative isolated singularity. Let $M_A$ be an MCM module. Then $\uEnd_A(M)$ is a right noetherian graded algebra.
\end{theorem}
\begin{proof} Set $B=\uEnd_A(M)$. Let $I$ be a graded right ideal of $B$ and let $S$ be a finite set consisting of homogeneous elements of $I$. We write $M^S=\sum_{b\in S} b(M)$. Then $M^S$ is a submodule of $M$. Consider the set $$\mathcal{X}=\{M^S|S\text{ finite set of homogeneous elements of }I\}.$$ Since $A$ is noetherian and $M$ is finitely generated, the set $\mathcal{X}$ has a maximal object. Let $M^{S_0}$ be a maximal object in $\mathcal{X}$. For every element $b'\in I$, consider the set $S_1=S_0\cup\{b'\}$. Since $M^{S_1}\supseteq M^{S_0}$, and by assumption $M^{S_0}$ is maximal, it follows $M^{S_1}=M^{S_0}$. Hence $b'(M)\subseteq \sum_{b\in S_0}b(M)$ and $M^{S_0}=IM$. Set $N=IM=M^{S_0}$.

Assume $S_0=\{b_1,\dots,b_n\}$, and the degrees of $b_1,\dots,b_n$ are $k_1,\dots,k_n$ respectively. Let $$\phi\colon M(-k_1)\oplus M(-k_2)\oplus\cdots \oplus M(-k_n)\longrightarrow N$$ be the homomorphism defined by the homomorphisms $\{b_1s^{-k_1},\dots, b_ns^{-k_n}\}$. Then $\phi$ is a homomorphism of degree 0.

We next prove that the right ideal $I$ is finitely generated. There two different situations.

Case 1. For $b'\in I$, assume the degree of the homomorphism $b'$ is $k$, and assume that the composition $M(-k)\overset{s^{-k}}\to M\overset{b'}\to N$ factors through a graded projective module, that is, there is a graded projective module $P$ such that $b's^{-k}$ is the composition $M(-k)\overset{r}\to P\overset{t}\to N$, where both $r$ and $t$ are homomorphisms of degree 0. Since $\phi$ is an epimorphism, there is a homomorphism $f\colon P\to N$ such that $t=\phi f$. Then $b's^{-k}=tr=\phi f r$. Let $h=fr$. Then $h$ is a morphism from $M(-k)$ to $M(-k_1)\oplus M(-k_2)\oplus\cdots \oplus M(-k_n)$. Let $h'\colon M\to M(-k_1)\oplus M(-k_2)\oplus\cdots \oplus M(-k_n)$ be the homomorphism such that $h=h's^{-k}$. Then $b'=\phi h'$. Let $p_i\colon M(-k_1)\oplus M(-k_2)\oplus\cdots \oplus M(-k_n)\to M(-k_i)$ be the projection map. Then $b'=\sum_{i=1}^n(b_is^{-k_i})(p_ih')$. For each $i$, let $h'_i=s^{-k_i}p_ih'$. Then $h'_i$ is an endomorphism of $M$. Hence, in this case, $b'=\sum_{i=1}^nb_ih'_i\in\sum_{i=1}^nb_iB$.

Case 2. Assume $b'$ does not factor through any projective modules. Since $A$ is a noncommutative isolated singularity, $\SHom_A(M,N)$ is finite dimensional by Lemma \ref{lem1}. Note that $\im(b)=bM\subseteq N$ for any $b\in I$, thus we may view $b$ as an element in $\Hom_A(M,N)$ and identify $I$ with a subspace of $\Hom_A(M,N)$. Then
\[\bar I= I/(I\cap\uPHom_A(M,N))\cong (I+\uPHom_A(M,N))/\uPHom_A(M,N)\] is finite dimensional, for the latter is a subspace of $\SHom_A(M,N)$. Choose homogeneous elements $f_1,\dots,f_m\in I$ such that their images $\overline{f}_1,\dots,\overline{f}_m$ in $\bar I$ form a basis of $\bar I$. Let $\overline{b'}$ be the image of $b'$ in the quotient space $\bar I$. Since $\overline{f}_1,\dots,\overline{f}_m$ is a basis, we may write $\overline{b'}=l_1\overline{f}_1+\cdots+l_m\overline{f}_m$ for some $l_1,\dots,l_m\in\kk$. Then $b'-(l_1f_1+\cdots+l_mf_m)\in I\cap \uPHom_A(M,N)$. By Case 1, $b'-(l_1f_1+\cdots+l_mf_m)=\sum_{i=1}^nb_ig_i$ for some $g_1,\dots,g_n\in B$. It follows that $b'=l_1f_1+\cdots+l_mf_m+\sum_{i=1}^nb_ig_i$.

Summarizing, the right ideal $I$ is generated by $f_1,\dots,f_m,b_1,\dots,b_n$.
\end{proof}

The proof of Theorem \ref{thm1} also implies the following result.

\begin{corollary} \label{cor1} Retain the same notations as in Theorem \ref{thm1}. Let $N$ be a finitely generated right graded $A$-module. Then $\uHom_A(M,N)$ is a finitely generated right graded $\uEnd_A(M)$-module.
\end{corollary}

We conclude the following result which suggests noncommutative resolutions for noncommutative singularities.

\begin{proposition} \label{prop1} Let $A$ be an AS-Gorenstein algebra and $M_A\in \gr A$ be an MCM module. Assume that $A$ is a noncommutative isolated singularity. Then $B:=\uEnd_A(M\oplus A)$ is a right noetherian graded algebra, and moreover, there is an equivalence of abelian categories $\qgr B\cong \qgr A$.
\end{proposition}
\begin{proof} The graded algebra $B$ may be written as a matrix algebra $$B=\left(
                                                             \begin{array}{cc}
                                                               \uEnd_A(M) & M \\
                                                               M^\vee & A \\
                                                             \end{array}
                                                           \right),
$$ where $M^\vee=\uHom_A(M,A)$. Define a map $$\varphi\colon M\otimes_AM^\vee\longrightarrow\uEnd_A(M), m_1\otimes_A f\mapsto [m_2\mapsto m_1f(m_2)].$$ The multiplication of the above matrix algebra reads as  $$\left(
                                                             \begin{array}{cc}
                                                               g_1 & m_1 \\
                                                               f_1 & a_1 \\
                                                             \end{array}
                                                           \right)\left(
                                                             \begin{array}{cc}
                                                               g_2 &m_2 \\
                                                               f_2 & a_2 \\
                                                             \end{array}
                                                           \right)=\left(
                                                             \begin{array}{cc}
                                                               g_1g_2+\varphi(m_1\otimes f_2) & g_1(m_2)+m_1a_2 \\
                                                               f_1g_2+a_1f_1 & f_1(m_2)+a_1a_2 \\
                                                             \end{array}
                                                           \right).$$
Let $e=\left(
         \begin{array}{cc}
           0 & 0 \\
           0 & 1 \\
         \end{array}
       \right).
$ Then $A\cong eBe$, and $BeB=\left(
                                                             \begin{array}{cc}
                                                               \varphi(M\otimes_AM^\vee) & M \\
                                                               M^\vee & A \\
                                                             \end{array}
                                                           \right)$. It is easy to see $\varphi(M\otimes_AM^\vee)=\uPHom_A(M,M)$. Then $$B/BeB\cong \SHom_A(M,M).$$
By Lemma \ref{lem1}, $B/BeB$ is finite dimensional. By the graded version of the proof of \cite[Lemma 2.3]{BHZ}, we obtain the desired equivalence $\qgr B\cong \qgr A$.
\end{proof}

\begin{lemma}\label{lem2} Let $A$ be an AS-Gorenstein algebra which is a noncommutative isolated singularity. Let $M_A$ be an MCM module. For  $X\in \gr A$, define $$\varphi_X\colon X\otimes_A\uHom_A(M,A)\longrightarrow\uHom_A(M,X),\ x\otimes f\mapsto[m\mapsto xf(m)].$$ Then both $\ker(\varphi)$ and $\coker(\varphi)$ are finite dimensional.
\end{lemma}
\begin{proof}
Let $0\to K\to P\to X\to 0$ be an exact sequence with $P$ a finitely generated graded projective $A$-module. Then we have the following commutative diagram
$$\xymatrix{
  &K\otimes_A\uHom_A(M,A)\ar[d]_{\varphi_K} \ar[r] & P\otimes_A\uHom_A(M,A) \ar[d]_{\varphi_P}\ar[r]&X\otimes_A\uHom_A(M,A)\ar[r]\ar[d]_{\varphi_X}&0 \\
  0\ar[r]&\uHom_A(M,K) \ar[r] & \uHom_A(M,P)\ar[r]&\uHom_A(M,X)&   }$$ with exact rows.
  Note that $\varphi_P$ is an isomorphism. By Snake Lemma, $\ker(\varphi_X)\cong \coker(\varphi_K)$.

  Since $\im(\varphi_X)=\uPHom_A(M,X)$, it follows from Lemma \ref{lem1} that $\coker(\varphi_X)$ is finite dimensional. Similarly $\coker(\varphi_K)$ and hence $\ker(\varphi_X)$ is also finite dimensional.
  \end{proof}

\begin{theorem}\label{thm2} Let $A$ be an AS-Gorenstein algebra with $\injdim A\ge2$ which is a noncommutative isolated singularity. Assume that $A$ is CM-finite. Let $\{P_0=A,P_1,\dots,P_n\}$ be the set of all the nonisomorphic indecomposable MCM modules (up to degree shifts). Let $M=\bigoplus_{i=1}^nP_n\oplus A$. Then $B:=\uEnd_A(M)$ is a right pre-resolution of $A$.
\end{theorem}
\begin{proof} By Theorem \ref{thm1}, $B$ is right noetherian. Since $M$ is an MCM $A$-module and $A$ is a noncommutative isolated singularity, it follows from Corollary \ref{cor1} that $\uHom_A(M,K)$ is finitely generated for every finitely generated right graded $A$-module $K$. Therefore $M$ is a small $A$-module.

Set $F=\uHom_A(M,-)\colon \gr A\to \gr B$ and $\mathcal F\colon \qgr A\to \qgr B$ to be the induced functor.
We next show that $\mathcal{F}$ is an equivalence. By the proof of Proposition \ref{prop1} (see also \cite[Lemma 2.3]{BHZ}), the functor $G=-\otimes_A M^\vee\colon\gr A\to \gr B$ induces an equivalence of abelian categories $\mathcal{G}\colon\qgr A\to \qgr B$. Now for any $X\in\gr A$, $$\mathcal{G}(\pi(X))=\pi(X\otimes_A M^\vee)=\pi(X\otimes_A \uHom_A(M,A)).$$ Let $\varphi_X\colon X\otimes_A\uHom_A(M,A)\longrightarrow\uHom_A(M,X)$ be the map as in Lemma \ref{lem2}. Then both $\ker(\varphi_X)$ and $\coker(\varphi_X)$ are finite dimensional, and $\varphi_X$ induces a natural isomorphism $$\pi(X\otimes_A\uHom_A(M,A))\cong\pi(\uHom_A(M,X)).$$ It follows that $\mathcal{F}$ is natural isomorphic to $\mathcal{G}$, and hence an equivalence.

By \cite[Theorem 5.4]{CKWZ} (where $A$ is assumed to be Cohen-Macaulay, but the proof applies in our case. See also \cite{Le} for commutative case), $\rgldim(B)=d$. Therefore all conditions in Definition \ref{def3} are satisfied. By definition $B$ is a right pre-resolution of $A$.
\end{proof}

\section{Noncommutative quadric hypersurfaces}\label{secthyper}

In this section, we focus on noncommutative resolutions of noncommutative quadric hypersurfaces. Let us recall some terminologies.

Let $A$ be a locally finite connected graded algebra. A graded $A$-module $M_A$ is called a {\it Koszul module} (cf. \cite{P}) if $M_A$ has a linear projective resolution:
$$\cdots \to P^{-n}\to\cdots\to P^{-1}\to P^0\to M\to0,$$ where $P^{-n}$ is a graded projective module generated in degree $n$ for all $n\ge0$. A left Koszul $A$-module is defined similarly. If the trivial module $\kk_A$ is a Koszul module, then $A$ is called a {\it Koszul Algebra}. It is known that a Koszul algebra $A$ must be quadratic, that is, $A$ may be written as $A=T(V)/(R)$, where $V$ is a finite dimensional vector space, and the generating relations $R$ is contained in $V\otimes V$. The {\it quadratic dual} of $A$ is defined to be the graded algebra $A^!=T(V^*)/(R^\bot)$, where $R^\bot$ is the orthogonal dual of $R$ in the space $V^*\otimes V^*$. Note that $A^!$ is also a Koszul algebra.

For a locally finite graded algebra $A$, the Hilbert series of $A$ is defined to be the formal power series: $$H_A(t)=\sum_{i\in \mathbb Z}(\dim A_i)t^i.$$

A noetherian connected graded algebra $S$ is called a {\it quantum polynomial algebra} if the following conditions are satisfied:
\begin{itemize}
  \item [(i)] $S$ is a Koszul AS-regular algebra;
  \item [(ii)] the Hilbert series of $S$ is $H_S(t)=\frac{1}{(1-t)^d}$ for some $d\ge1$.
\end{itemize}

Let $S$ be a quantum polynomial algebra. Suppose $w\in S_2$ is a central regular element in $S$ of degree two. The quotient algebra $A=S/Sw$ is usually called a {\it noncommutative quadric hypersurface}.

The following properties of noncommutative quadric hypersufaces are well known (see \cite[Lemma 5.1(1)]{SvdB}, \cite[Lemma 1.2]{HY} for instance). Note that for a quantum polynomial algebra the Gorenstein parameter coincides with the global dimension.

\begin{lemma}\label{qlem1} Assume $S$ is a quantum polynomial algebra with global dimension $d+1$ ($d\ge0$). Let $w\in S_2$ be a central regular element of $S$, and let $A=S/Sw$.
\begin{itemize}
  \item [(i)]  $A$ is a Koszul algebra.
  \item [(ii)] $A$ is AS-Gorenstein of injective dimension $d$ with Gorenstein parameter $d-1$.
  \item [(iii)] There is a central regular element $\varpi \in A^!_2$ such that $S^!\cong A^!/A^!\varpi$.
\end{itemize}
\end{lemma}

\begin{setup}\label{setup} In the rest of this section, $S$ is a quantum polynomial algebra of global dimension $d+1$ with $d\ge0$, and $w\in S_2$ is a central regular element. Set $A=S/Sw$.
\end{setup}

We recall some results obtained in \cite{SvdB}. Let $D^b(\gr A)$ be the bounded derived category of $\gr A$. There is a Koszul duality (cf. \cite[Subsection 2.4]{SvdB}, or \cite[Section 3]{BGS} for general situation):
\begin{equation}\label{eqq1}
  K\colon D^b(\gr A)\longrightarrow D^b(\gr A^!).
\end{equation}
Notice that
\begin{equation}\label{eqq7}
  K(\kk)=A^!,\ K(A)=\kk\text{ and }K(M(1))=K(M)[-1](1)\text{ for all }M\in D^b(\gr A).
\end{equation}

The above duality induces the following duality:
\begin{equation}\label{eqq2}
  \overline{K}\colon D^b(\gr A)/per A\longrightarrow D^b(\gr A^!)/D^b_{\tor} (\gr A^!),
\end{equation}
where $per A$ is the full subcategory of $D^b(\gr A)$ consisting of perfect complexes, and $D^b_{\tor}(\gr A^!)$ is the full subcategory of $D^b(\gr A^!)$ consisting of complexes with finite dimensional total cohomology. Notice that $D^b(\gr A^!)/D^b_{\tor} (\gr A^!)\cong D^b(\qgr A^!)$. Let $\mcm A$ be the full subcategory of $\gr A$ consisting of all the MCM modules, and let $\underline{\mcm} A$ be the stable category. Then there is a natural equivalence of triangulated categories \cite[Theorem 4.4.1(2)]{B}:
\begin{equation}\label{eqq3}
  G\colon\underline{\mcm}A\longrightarrow D^b(\gr A)/per A.
\end{equation}
Combining the functors in (\ref{eqq2}) and (\ref{eqq3}), we obtain the following Buchweitz's duality (cf. \cite[Theorem 3.2]{SvdB}):
\begin{equation}\label{eqq4}
  B\colon\underline{\mcm}A\longrightarrow D^b(\qgr A^!).
\end{equation}

By Lemma \ref{qlem1}(iii), there is a central regular element $\varpi\in A^!_2$. Let $A^![\varpi^{-1}]$ be the localization of $A^!$ by the multiplication system defined by $\varpi$. Then $A^![\varpi^{-1}]$ is a $\mathbb Z$-graded algebra. Define a finite dimensional algebra (cf. \cite[Subsection 5.1]{SvdB}):
\begin{equation}\label{def-CA}
  C(A)=A^![\varpi^{-1}]_0,
\end{equation}
which is the degree zero part of $A^![\varpi^{-1}]$. Since $S$ is a quantum polynomial algebra, it follows that (cf. \cite{SvdB}):
\begin{equation}\label{eqqdim}
  \dim C(A)=\sum_{i\ge0}\dim S^!_{2i}=\frac{1}{2}\dim S^!.
\end{equation}
Notice that $A^!/A^!\varpi$ is finite dimensional. It follows that $M\in \tor A^!$ if and only if any $m\in M$ is annihilated by some power of $\varpi$. Therefore, there is an equivalence: $$L\colon D^b(\qgr A^!)\longrightarrow D^b(\mmod C(A)),$$
where $\mmod C(A)$ is the category of finite dimension modules. Notice that
 \begin{equation}\label{eqq8}
   L(\pi(A^!))=C(A).
\end{equation} Combining $L$ with the Buchweitz's duality, we obtain the following equivalence (cf. \cite[Proposition 5.2]{SvdB}):
\begin{equation}\label{eqq5}
  F\colon \underline{\mcm} A\overset{B}\longrightarrow D^b(\qgr A^!)\overset{L}\longrightarrow D^b(\mmod C(A)).
\end{equation}

We may put the above mentioned triangle equivalences in the following commutative diagram
$$
\xymatrix
{
\underline{\mcm} A \ar[r]^{B} \ar@(ur,ul)[rr]^{F} \ar[d]^{G} & D^b(\qgr A^!) \ar[d]^{\cong} \ar[r]^{L} &D^b(\mmod C(A)) \\
D^b(\gr A)/per A \ar[r]^{\overline{K}} & D^b(\gr A^!)/D^b_{\tor} (\gr A^!)
}.
$$

We remark that the finite dimensional algebra $C(A)$ may be obtained from a Clifford deformation of the Frobenius algebra $S^!$ (cf. \cite{HY}). In this section, we will give a new method to obtain the finite dimensional algebra $C(A)$.

Taking a minimal graded projective resolution of $\kk_A$ as follows:
\begin{equation}\label{res}
  \cdots\longrightarrow P^{-d}\overset{\partial^{-d}}\longrightarrow P^{-d+1}\overset{\partial^{-d+1}}\longrightarrow\cdots \longrightarrow P^0\longrightarrow\kk_A\longrightarrow0.
\end{equation}
Let $\Omega^d(\kk_A)=\ker\partial^{-d+1}$ be the $d$th syzygy of the trivial module. Since $A$ is a Koszul algebra, $\Omega^{d}(\kk_A)$ is generated in degree $d$. We fix a notion as follows:
$$\mathbb M:=\Omega^d(\kk_A)(d).$$
The following properties of $\mathbb M$ is clear.

\begin{lemma} Retain the notation as above.
\begin{itemize}
  \item [(i)] $\mathbb M$ is a Koszul $A$-module.
  \item [(ii)] $\mathbb M$ is an MCM module.
\end{itemize}
\end{lemma}

We investigate more properties of $\mathbb M$. For a right graded $A$-module $X$, let $X^\vee =\uHom_A(X,A)$ denote the dual module. Clearly $X^\vee$ is a left graded module in an obvious way.

\begin{proposition}\label{qlem2} Retain the notations as above. Then $\mathbb M^\vee(1)$ is a left Koszul $A$-module.
\end{proposition}
\begin{proof} Applying the functor $\uHom_A(-,A)$ to the exact sequence (\ref{res}), we obtain the sequence
\begin{equation}\label{ex1}
  0\to \uHom_A(P^0,A)\to\uHom_A(P^{-1},A)\to\cdots\to\uHom_A(P^{-d+1},A)\overset{\iota^\vee}\longrightarrow\uHom_A(\Omega(\kk_A),A)\to0.
\end{equation}
Since $A$ is AS-Gorenstein of injective dimension $d$ with Gorenstein parameter $d-1$, the sequence (\ref{ex1}) is exact except at the last position, where
the cohomology group is ${}_A\kk(d-1)$. Note that $\mathbb M^\vee(d)=\uHom_A(\Omega(\kk_A),A)$. The sequence (\ref{ex1}) implies an exact sequence of left $A$-modules
\begin{equation}\label{ex2}
  0\to K\to\mathbb M^\vee(d)\to {}_A\kk(d-1)\to0,
\end{equation}
where $K=\im\iota^\vee$. Note that $K$ has a projective resolution $$0\to \uHom_A(P^0,A)\to\uHom_A(P^{-1},A)\to\cdots\to\uHom_A(P^{-d+1},A)\to K\to0.$$ Hence $K(-d+1)$ is a left Koszul $A$-module, and the Koszulity of $\mathbb M^\vee(1)$ follows from (\ref{ex2}).
\end{proof}

\begin{proposition}\label{qlem3} Retain the notations as above. We have the following properties.
\begin{itemize}
\item[(i)] $\uExt_A^i(\kk_A,\mathbb M)=0$ for $i<d$.
\item[(ii)] The graded vector space $\uExt_A^d(\kk_A, \mathbb M)$ is concentrated in degree $-d$.
\end{itemize}
\end{proposition}
\begin{proof} (i) follows from the fact that $\mathbb M$ is an MCM module.

(ii) Since $A$ is AS-Gorenstein, $R\uHom_A(-,A)\colon D^b(\gr A)\to D^b(\gr A^\circ)$ is a duality (cf. \cite{Y}), where $A^\circ$ is the opposite algebra of $A$. Since $\mathbb M$ is an MCM module, it follows that $R\uHom_A(\mathbb M,A)\cong \uHom_A(\mathbb M,A)=\mathbb M^\vee$ in $D^b(\gr A^\circ)$. Note that $R\uHom_A(\kk_A,A)\cong {}_A\kk[-d](d-1)$. Therefore we have
\begin{eqnarray*}
 \uExt_A^d(\kk_A,\mathbb M)_i&=&\Hom_{D^b(\gr A)}(\kk_A,\mathbb M[d](i))\\
 &\cong&\Hom_{D^b(\gr A^\circ)}\left(R\uHom_A(\mathbb M,A)[-d],R\uHom_A(\kk_A,A)(i)\right)\\
 &\cong&\Hom_{D^b(\gr A^\circ)}\left(\mathbb M^\vee,{}_A\kk(d+i-1)\right)\\
 &\cong&\Hom_{D^b(\gr A^\circ)}\left(\mathbb M^\vee(1),{}_A\kk(d+i)\right)\\
 &\cong&\Hom_{\gr A^\circ}\left(\mathbb M^\vee(1),{}_A\kk(d+i)\right).
\end{eqnarray*}
By Proposition \ref{qlem2}, $\mathbb M^\vee(1)$ is generated in degree one. Therefore, $\Hom_{\gr A^\circ}(\mathbb M^\vee(1),{}_A\kk(d+i))=0$ for $i\neq -d$.
Hence $\uExt_A^d(\kk_A, \mathbb M)$ is concentrated in degree $-d$.
\end{proof}

\begin{theorem}\label{qthm1} Retain the notations as above. We have
\begin{itemize}
  \item [(i)] $\End_{\gr A}(\mathbb M)\cong C(A)$;
  \item [(ii)] $A$ is a noncommutative isolated singularity if and only if $\End_{\gr A}(\mathbb M)$ is semisimple.
\end{itemize}
\end{theorem}
\begin{proof} (i) From the resolution (\ref{res}), we have the following exact triangle in $D^b(\gr A)$: $$\Omega^d(\kk_A)[d-1]\to P^\cdot\to \kk_A\to \Omega^d(\kk_A)[d]),$$ where $P^\cdot$ is the complex $0\to P^{-d+1}\overset{\partial^{-d+1}}\longrightarrow\cdots\overset{\partial^1}\longrightarrow P^0\to0$. Therefore
\begin{equation}\label{eqq6}
  \Omega^d(\kk_A)[d]\cong \kk_A
\end{equation}
in the quotient category $D^b(\gr A)/per A$. Let $F$ be the equivalence as in (\ref{eqq5}). Considering the equivalence functors (\ref{eqq1})--(\ref{eqq4}), we have $F(\Omega^d(\kk_A)[d])=LB(\Omega^d(\kk_A)[d])=L\overline{K}G(\Omega^d(\kk_A)[d])$.  By (\ref{eqq6}), $G(\Omega^d(\kk_A)[d])\cong \kk_A$. By (\ref{eqq1}) and (\ref{eqq7}), we have $F(\Omega^d(\kk_A)[d])\cong L\overline{K}(\kk_A)\cong L(\pi(A^!))$. By (\ref{eqq8}), $L(\pi(A^!))\cong C(A)$. Finally, we obtain that
\begin{equation}\label{eqqf}
  F(\Omega^d(\kk_A)[d])\cong C(A).
\end{equation}
Since $F$ is an equivalence, we have $$\End_{\underline{\mcm}(A)}(\Omega^d(\kk_A)[d])\cong \End_{D^b(\mmod C(A))}(C(A))\cong C(A),$$ whereas in the triangulated category $\underline{\mcm}A$, we have $$\End_{\underline{\mcm}(A)}(\mathbb M)= \End_{\underline{\mcm}(A)}(\Omega^d(\kk_A)(d)) \cong \End_{\underline{\mcm}(A)}(\Omega^d(\kk_A))\cong\End_{\underline{\mcm}(A)}(\Omega^d(\kk_A)[d]).$$ Therefore, $$\End_{\underline{\mcm}(A)}(\mathbb M)\cong C(A).$$
By Proposition \ref{qlem2}, $\mathbb M^\vee(1)$ is a Koszul module. In particular, $\mathbb M^\vee$ is generated in degree 1 and $\Hom_{\gr A}(\mathbb M,A)=0$, thus $\End_{\underline{\mcm}A}(\mathbb M)=\End_{\gr A}(M)$ and the desired isomorphism (i) follows.

The statement (ii) follows from \cite[Theorem 6.3]{HY} (see also, \cite[Theorem 4.13]{MU}).
\end{proof}

Since $A$ is a Koszul algebra, we may compute $\mathbb M$ and $\End_{\gr A}(\mathbb M)$ by Koszul resolution of $A$. Now assume $A=T(V)/(R)$ for some finite dimension vector space $V$ with generating relations $R\subseteq V\otimes V$. Let $C_0=\kk$, $C_1=V$, $C_2=R$ and $C_n=\bigcap_{i+j=n-2}V^{\otimes i}\otimes R\otimes V^{\otimes j}$ for $n>2$. Then the minimal projective resolution of $\kk_A$ reads as follows (cf. \cite{BGS}):
\begin{equation}\label{res2}
  \cdots \longrightarrow C_n\otimes A\overset{\partial^{-n}}\longrightarrow C_{n-1}\otimes A\overset{\partial^{n-1}}\longrightarrow\cdots\overset{\partial^{-1}}\longrightarrow C_0\otimes A\longrightarrow\kk_A\longrightarrow0,
\end{equation}
where the differential is defined as follows: if $\sum_{i=1}^mx_{1i}\otimes\cdots \otimes x_{ni}\in C_n$, for every $a\in A$, $$\partial^{-n}((\sum_{i=1}^mx_{1i}\otimes\cdots \otimes x_{ni})\otimes a)=\sum_{i=1}^m(x_{1i}\otimes\cdots \otimes x_{n-1,i})\otimes x_{ni}a.$$
By the above resolution, $\Omega^d(\kk_A)\cong \im\partial^{-d}$. From the exact sequence $$0\to \im\partial^{-d-1}\to C_d\otimes A\to\Omega^d(\kk_A)\to0,$$ we obtain
\begin{equation}\label{eqqend1}
  \End_{\gr A}(\Omega^d(\kk_A))=\{f\in\End_{\gr A}(C_d\otimes A)|f(\im\partial^{-d-1})\subseteq\im \partial^{-d-1}\}.
\end{equation}
 Notice that the restriction of $\partial^{-d-1}$ on $(C_{d+1}\otimes A)_{d+1}\cong C_{d+1}$ is injective, and every element $f\in \End_{\gr A}(C_d\otimes A)$ is defined by its restriction on $C_d$. It follows an isomorphism
\begin{equation}\label{eqqend2}
  \{f\in\End_{\gr A}(C_d\otimes A)|f(\im\partial^{-d-1})\subseteq\im \partial^{-d-1}\}\cong \{f\in\End_{\kk}(C_d)|(f\otimes 1)(C_{d+1})\subseteq C_{d+1}\}.
\end{equation}

Combining (\ref{eqqend1}) and (\ref{eqqend2}), we have the following isomorphism.

\begin{proposition} \label{prop2} Write $A=T(V)/(R)$. We have $$\End_{\gr A}(\mathbb M)\cong\{f\in\End_{\kk}(C_d)|(f\otimes 1)(C_{d+1})\subseteq C_{d+1}\},$$ where $C_n=\bigcap_{i+j=n-2}V^{\otimes i}\otimes R\otimes V^{\otimes j}$ for $n=d,d+1$, and $f\otimes1$ is viewed as a linear map in $\End_\kk(C_d\otimes V)$.
\end{proposition}

\begin{remark} The proposition above provides a relatively easy way to compute the endomorphism ring of $\mathbb M$, especially when $d$ is small. We will compute a detailed example of noncommutative quadric hypersurface of dimension 2 in the next section. According to Theorem \ref{qthm1} and the proposition above, to find whether $A$ is a noncommutative isolated singularity is a linear algebra problem.
\end{remark}

Let us check the MCM modules when $A$ is a noncommutative isolated singularity.

\begin{lemma}\label{lem-ind} Assume that $A$ is a noncommutative isolated singularity. Each nonprojective indecomposable MCM $A$-module is isomorphic to a direct summand of $\mathbb M$ (up to a degree shift).
\end{lemma}
\begin{proof} By the isomorphism (\ref{eqqf}) in the proof of Theorem \ref{qthm1}, $F(\Omega^d(\kk_A))\cong C(A)[-d]$. Since $C(A)$ is semisimple and $F$ is an equivalence, all the indecomposable objects in $\underline{\mcm} A$ are direct summands of $\Omega^d(\kk_A)$ (up to degree shifts). Notice that the class of nonprojective indecomposable MCM module over $A$ is in one-to-one correspondence to the class of indecomposable objects in $\underline{\mcm}A$ (cf. \cite[Lemma 3.4]{SvdB}). Since $\mathbb M$ is a shift of $\Omega^d(\kk_A)$, the result follows.
\end{proof}

We have the following properties of indecomposable MCM modules.

\begin{proposition}\label{prop-M} Retain the notation as above and keep Setup \ref{setup}. Assume that $A$ is a noncommutative isolated singularity.
\begin{itemize}
  \item [(i)] For $n\ge d$, $\dim\Omega^n(\kk_A)_{n}=\frac{1}{2}\dim S^!$, where $\Omega^n(\kk_A)_n$ is the degree $n$ part of the graded module $\Omega^n(\kk_A)$.
  \item [(ii)] $\dim \mathbb M_0=\dim\End_{\gr A}(\mathbb M)$.
  \item [(iii)] Suppose that $\End_{\gr A}(\mathbb M)$ is a direct product of $\kk$. If $N$ is an indecomposable MCM module, then $N(i)\cong A/xA$ for some $i\in \mathbb Z$ and some element $x\in A_1$. Moreover, $\Omega(\mathbb M)\cong\mathbb M(-1)$.
\end{itemize}
\end{proposition}
\begin{proof} (i) By Lemma \ref{qlem1}(ii), we have the following exact sequence $$0\longrightarrow A^!(-2)\overset{\cdot\varpi}\longrightarrow A^!\longrightarrow S^!\longrightarrow0.$$ Then we have $\dim A^!_i=\dim A^!_{i-2}+\dim S^!_i$ for all $i\ge2$, and $\dim A^!_i=\dim S^!_i$ for $i=0,1$. Then by an iterative computation, we have
$$\dim A^!_n=\left\{
               \begin{array}{ll}
                 \dim S_0^!+\dim S_2^!+\cdots+S^!_n, & \hbox{when $n$ is even;} \\
                 \dim S_1^!+\dim S_3^!+\cdots+S^!_n, & \hbox{when $n$ is odd.}
               \end{array}
             \right.
$$
Since $S$ is a quantum polynomial algebra, $H_{S^!}(t)=(1+t)^{d+1}$. Therefore $\dim A^!_n=\frac{1}{2}\dim S^!$  for $n\ge d$. Since $A$ is a Koszul algebra and $\dim A^!_n=\dim \Hom_{\gr A}(\Omega^n(\kk_A)(n),\kk_A)=\dim \Omega^n(\kk_A)_n$.

(ii) By Theorem \ref{qthm1}, $\End_{\gr A}(\mathbb M)\cong C(A)$. Hence $\dim\End_{\gr A}(\mathbb M)=\dim C(A)=\frac{1}{2}\dim S^!$ (see Equation (\ref{eqqdim})). Since $\mathbb M=\Omega^d(\kk_A)(d)$, the identity follow from (i).

(iii) By Lemma \ref{lem-ind}, it suffices to show that the result holds for each indecomposable direct summand $N$ of $\mathbb M$.

Assume $\mathbb M=\mathbb M^1\oplus\cdots\oplus\mathbb M^s$, where each $\mathbb M^i$ is indecomposable. By assumption $\End_{\gr A}(\mathbb M)$ is a direct product of $\kk$, which forces that $\End_{\gr A}(\mathbb M_i)=\kk$ for all $i$ and $\Hom_{\gr A}(\mathbb M_i, \mathbb M_j)=0$ for $i\ne j$. Hence $s=\dim \End_{\gr A}(\mathbb M)$, and $\mathbb M$ has no projective direct summands, otherwise if some $\mathbb M_i$ is projective, then $\Hom_{\gr A}(\mathbb M_i, \mathbb M_j)\ne 0$ for any $j$. Since $\mathbb M$ is a Koszul module, each $\mathbb M^i$ is a Koszul module. Hence $\mathbb M^i$ is generated in degree 0 for every $i$. By (i) and (ii), $s=\dim \mathbb M_0=\frac{1}{2}\dim S^!$. Therefore, each $\mathbb M^i$ is a cyclic module. Hence we have the following exact sequence $$0\to\Omega(\mathbb M^i)\to A\to \mathbb M^i\to0,\text{ for all }1\leq i\leq s.$$
Then $\Omega^{d+1}(\kk_A)(d)\cong \Omega(\mathbb M)\cong \bigoplus_{i=1}^s\Omega(\mathbb M^i)$. By (i), $\dim\Omega^{d+1}(\kk_A)_{d+1}=s$. Hence $\dim\Omega(\mathbb M^i)_1=1$ for all $1\leq i\leq s$ since each $\mathbb M^i$ is not projective. Note that $\Omega(\mathbb M^i)$ is generated in degree 1. It follows that there is an element $x_i\in A_1$ such that $\Omega(\mathbb M^i)\cong x_iA$ for all $1\leq i\leq s$. Hence $\mathbb M^i\cong A/x_iA$.

Since $\dim \End_{\gr A}(\mathbb M)=\dim \mathbb M_0$, $\Hom_{\gr A}(\mathbb M^i,\mathbb M^j)=0$ if $i\neq j$. Hence $\mathbb M^i\ncong\mathbb M^j$ if $i\neq j$, which implies $\Omega(\mathbb M^i)\ncong\Omega(\mathbb M^j)$ for $i\neq j$ since $\Omega=[-1]$ is the suspension functor in the triangulated category $\underline{\mcm}A$. Since each $\Omega(\mathbb M^i)$ is indecomposable and generated in degree 1, it follows that the set $\{\Omega(\mathbb M^1),\dots,\Omega(\mathbb M^s)\}=\{\mathbb M^1(-1),\dots,\mathbb M^s(-1)\}$ by Lemma \ref{lem-ind}. Therefore $\Omega(\mathbb M)\cong \mathbb M(-1)$.
\end{proof}

Lemma \ref{lem-ind} also provides a way to construct noncommutative resolution of noncommutative isolated singularities.

\begin{theorem}\label{thm-res-iso} Keep the notions in Setup \ref{setup}. Assume that $A$ is a noncommutative isolated singularity.
Then $B=\uEnd_A(\mathbb M\oplus A)$ is a right pre-resolution of $A$. Moreover, $B$ is concentrated in nonnegative degrees.
\end{theorem}
\begin{proof} That $B$ is a right pre-resolution of $A$ follows from Theorem \ref{thm2} and Lemma \ref{lem-ind}.
Since $\mathbb M$ is a Koszul module, $B$ is concentrated in nonnegative degrees.
\end{proof}

\begin{remark}
The algebra $B$ is isomorphic to the matrix algebra $\left(
                                                       \begin{array}{cc}
                                                         \uEnd_A(\mathbb M) & \mathbb M \\
                                                         \mathbb M^\vee & A \\
                                                       \end{array}
                                                     \right)
$. By Proposition \ref{qlem2}, $\mathbb M^\vee$ is concentrated in degrees not less than 1. By Proposition \ref{qlem3}(ii), every element $f\in\uEnd_A(\mathbb M)_{\ge 1}$ factors through a projective module. By \cite[Proposition 7.5.1]{McR},
$$B_0= \left(
                                                       \begin{array}{cc}
                                                         \End_{\gr A}(\mathbb M) & \mathbb M_0 \\
                                                         0 & \kk \\
                                                       \end{array}
                                                     \right)$$
is of global dimension 1, since $\End_{\gr A}(\mathbb M)$ is semisimple by assumption.
 In the next section, we give a concrete example with detailed computations of elements of $B$.
\end{remark}

\section{An example}

In this section, we give a detailed computation of indecomposable MCM module of an explicit noncommutative quadric hypersurfaces. Let $\kk=\mathbb C$, let $S=\kk\langle x,y,z\rangle/(R)$, where $R=\text{span}\{xz+zx,yz+zy, x^2+y^2\}$. Then $S$ is a quantum polynomial algebra of global dimension 3, which is an AS-regular algebra of type $S_2$ as listed in \cite[Table 3.11, P.183]{AS}.

The following facts were proved in \cite[Setion 9]{HY}, see also \cite{Hu} for a complete classification of noncommutative conics.
\begin{lemma}
Let $\varpi=x^2+z^2\in S_2$. Then
\begin{itemize}
  \item [(i)] $\varpi$ is a central regular element of $S$,
  \item [(ii)] $A=S/S\varpi$ is a noncommutative isolated singularity ,
  \item [(iii)] $C(A)\cong \kk^4$, where $C(A)$ is the algebra defined in (\ref{def-CA}).
\end{itemize}
\end{lemma}

Consider the Koszul resolution
$$\cdots\longrightarrow (R\otimes V\cap V\otimes R)\otimes A \overset{\partial^{-3}}\longrightarrow R\otimes A\overset{\partial^{-2}}\longrightarrow V\otimes A\overset{\partial^{-1}}\longrightarrow A\longrightarrow\kk_A\longrightarrow0$$  of $\kk_A$, where $V=\text{span}\{x,y,z\}$.
Clearly $R$ has a basis
\[ x\otimes z + z\otimes x,\ y\otimes z+z\otimes y,\ x\otimes x+ y\otimes y,\ x\otimes x + z\otimes z,
\]
and by direct calculation, $R\otimes V\cap V\otimes R$ has a basis
\begin{gather*}
(x\otimes z + z\otimes x)\otimes x+ (y\otimes z + z\otimes y)\otimes y+(x\otimes x + y\otimes y)\otimes z, \\
2(x\otimes z + z\otimes x)\otimes x+ (y\otimes z + z\otimes y)\otimes y+(x\otimes x + y\otimes y)\otimes z +  (x\otimes x + z\otimes z)\otimes z, \\
(y\otimes z + z\otimes y)\otimes z -(x\otimes x + y\otimes y)\otimes y +  (x\otimes x + z\otimes z)\otimes y, \\
(x\otimes z + z\otimes x)\otimes z+ (y\otimes z + z\otimes y)\otimes z -(x\otimes x + y\otimes y)\otimes y +  (x\otimes x + z\otimes z)\otimes (x+y).
\end{gather*}

Set $\mathbb M=\Omega^2(\kk_A)(2)= \im \partial^{-2}(2)$.
 Then $\mathbb M$ is a Koszul module which is generated by $\frac{1}{2}\dim S^!(= 4)$ elements, see Proposition \ref{prop-M}(i). The generating relations of $\mathbb M$ is equal to $\im \partial^{-3}(2)$ which is also generated by 4 elements. Thus we may write $\mathbb M$ as a quotient module of a free module:
$$0\longrightarrow K\longrightarrow m_1A\oplus m_2A\oplus m_3A\oplus m_4A\longrightarrow\mathbb M\longrightarrow0,$$ where  $\{m_1,m_2,m_3,m_4\}$ is a free basis, and $K$ is the submodule generated by \begin{eqnarray*}
                                                          r_1 &=& m_1x+m_2y+m_3z, \\
                                                          r_2 &=& 2m_1x+m_2 y+m_3z+m_4z,\\
                                                          r_3 &=& m_2z-m_3y+m_4y,\\
                                                          r_4 &=& m_1z+m_2z-m_3y+m_4(x+y).
                                                        \end{eqnarray*}

To find indecomposable MCM $A$-modules, we only need to find a set of primitive idempotents of $\End_{\gr A}(\mathbb M)$ by Lemma \ref{lem-ind}. Let $F=m_1A\oplus m_2A\oplus m_3A\oplus m_4A$. Note that degree one part of $K$ is $K_1=\text{span}\{r_1,r_2,r_3,r_4\}$. We have $$\End_{\gr A}(\mathbb M)=\{\theta\in\End_{\gr A}(F)|\theta(r_i)\in K_1, i=1,2,3,4\}.$$
By some computations on linear equations, we have
$$\End_{\gr A}(\mathbb M)=\left\{\left(
                            \begin{array}{cccc}
                              b+d & 0 & a & a \\
                              0 & b & c & 0 \\
                              0 & -c & b & 0 \\
                              a & c & d & b+d \\
                            \end{array}
                          \right):a,b,c,d\in\kk\right\}.
$$
We have the following complete set of primitive idempotents in $\End_{\gr A}(\mathbb M)$:
$$e_1=\left(
                            \begin{array}{cccc}
                              0 & 0 & 0 & 0 \\
                              0 & \frac{1}{2} & \frac{1}{2}i & 0 \\
                              0 & -\frac{1}{2}i & \frac{1}{2} & 0 \\
                              0 & \frac{1}{2}i & -\frac{1}{2} & 0 \\
                            \end{array}
                          \right),\quad
                          e_2=\left(\begin{array}{cccc}
                              0 & 0 & 0 & 0 \\
                              0 & \frac{1}{2} & -\frac{1}{2}i & 0 \\
                              0 & \frac{1}{2}i & \frac{1}{2} & 0 \\
                              0 & -\frac{1}{2}i & -\frac{1}{2} & 0 \\
                            \end{array}
                          \right),$$
$$e_3=\left(
                            \begin{array}{cccc}
                              \frac{1}{2} & 0 & \frac{1}{2} & \frac{1}{2} \\
                              0 & 0 & 0 & 0 \\
                              0 & 0 & 0 & 0 \\
                              \frac{1}{2} & 0 & \frac{1}{2} & \frac{1}{2} \\
                            \end{array}
                          \right),\quad
                          e_4=\left(\begin{array}{cccc}
                              \frac{1}{2} & 0 & -\frac{1}{2} & -\frac{1}{2} \\
                              0 & 0 & 0 & 0 \\
                              0 & 0 & 0 & 0 \\
                             -\frac{1}{2} & 0 & \frac{1}{2} & \frac{1}{2} \\
                            \end{array}
                          \right).$$

Therefore, we have the following nonprojective nonisomorphic indecomposable MCM modules:
$$\begin{array}{ll}
  \mathbb M^1&=\mathbb Me_1=(\overline{m}_2-i\overline{m}_3+i\overline{m}_4)A,\\
  \mathbb M^2&=\mathbb Me_2=(\overline{m}_2+i\overline{m}_3-i\overline{m}_4)A,\\
  \mathbb M^3&=\mathbb Me_3=(\overline{m}_1+\overline{m}_4)A,\\
  \mathbb M^4&=\mathbb Me_4=(\overline{m}_1-\overline{m}_4)A,\\
\end{array}$$
where $\overline{m}_1,\overline{m}_2,\overline{m}_3,\overline{m}_4$ are the image of $m_1,m_2,m_3,m_4$ in $\mathbb M$, and $i=\sqrt{-1}$ is a square root of $-1$.

Since $\mathbb M^1,\dots,\mathbb M^4$ are Koszul modules, a straightforward check shows that
$$\mathbb M^1\cong A/(y+iz)A,\ \mathbb M^2\cong A/(y-iz)A,\ \mathbb M^3\cong A/(x+z)A,\ \mathbb M^4\cong A/(x-z)A.$$
Note that $y+iz,y-iz,x+z,x-z$ are nilpotent elements in $A$. Also, we have $$A/(y+iz)A\cong (y+iz)A(1),\ A/(y-iz)A\cong (y-iz)A(1),$$ $$A/(x+z)A\cong (x+z)A(1),\ A/(x-z)A\cong (x-z)A(1).$$

Summarizing, we have the following conclusion.
\begin{proposition} The quadric hypersurface $A$ has nonprojective nonisomorphic indecomposable MCM modules (up to degree shifts):
$$\begin{array}{ll}
  \mathbb M^1&\cong A/(y+iz)A\cong (y+iz)A(1),\\
  \mathbb M^2&\cong A/(y-iz)A\cong (y-iz)A(1),\\
  \mathbb M^3&\cong A/(x+z)A\cong (x+z)A(1),\\
  \mathbb M^4&\cong A/(x-z)A\cong (x-z)A(1).\\
\end{array}$$
\end{proposition}

By Theorem \ref{thm-res-iso}, $A$ has a right pre-resolution $\uEnd_A(\mathbb M\oplus A)$. Write $u_1=y+iz,u_2=y-iz,u_3=x+z,u_4=x-z$.
Since $A/u_iA$ is an MCM module for each $1\leq i\leq 4$, the map $\tau^\vee\colon \uHom_A(A,A)\to \uHom_A(u_iA,A)$ induced from the inclusion map $\tau\colon u_iA\to A$ is surjective. Therefore, for each homogeneous element $f\in\uHom_A(u_iA,A)$, there is a homogeneous element $a\in A$ such that $f(u_i)=au_i$. On the other hand, for each homogeneous element $a\in A$, there is a graded right $A$-module morphism $f\colon u_iA\to A$ defined by $f(u_i)=au_i$. Hence, we obtain $\uHom_A(u_iA,A)\cong Au_i(1)$.

Since $\underline{\mcm}A$ is semisimple and $C(A)\cong\End_A(\mathbb M)\cong\kk^4$, by the equivalence functor (\ref{eqq5}), we have $\uExt^1_A(u_iA,u_jA)=0$ for all $i\neq j$ and $\uExt^1_A(u_iA,u_iA)\cong \kk$. Therefore, the exact sequence $0\to u_iA\to A\to A/u_iA\to0$ induces a surjective map $\uHom_A(A,u_jA)\to\uHom_A(u_iA,u_jA)$ for $i\neq j$. Then we obtain  $\uHom_A(u_iA,u_jA)=u_jAu_i(1)$ for $i\neq j$. Similarly, we have $\uEnd_A(u_iA)=\kk\oplus u_iAu_i(1)$ for $i=1,2,3,4$. Then $\uEnd_A(\mathbb M\oplus A)$ is isomorphic to the following algebra
$$\left(
                                    \begin{array}{ccccc}
                                      u_1Au_1(1) & u_1Au_2(1) & u_1Au_3(1) & u_1Au_4(1) & u_1A(1) \\
                                      u_2Au_1(1) & u_2Au_2(1) & u_2Au_3(1) & u_2Au_4(1) & u_2A(1) \\
                                      u_3Au_1(1) & u_3Au_2(1) & u_3Au_3(1) & u_3Au_4(1) & u_3A(1) \\
                                      u_4Au_1(1) & u_4Au_2(1) & u_4Au_3(1) & u_4Au_4(1) & u_4A(1) \\
                                      Au_1 & Au_2 & Au_3 & Au_4 & A_{\ge1} \\
                                    \end{array}
                                  \right)\bigoplus\left(
                                                    \begin{array}{ccccc}
                                                      \kk & 0 & 0 & 0 & 0 \\
                                                      0 & \kk & 0 & 0 & 0 \\
                                                      0 & 0 & \kk & 0 & 0 \\
                                                      0 & 0 & 0 & \kk & 0 \\
                                                      0 & 0 & 0 & 0 & \kk \\
                                                    \end{array}
                                                  \right),
$$
where the multiplication is defined as below. For consistency of notations we set $u_5=1$. We simply write elements in the left matrix as $\left(u_ia_{ij}u_j\right)$. Then
$$\left(u_ia_{ij}u_j\right)\left(u_ib_{ij}u_j\right)=\left(\sum_{k=1}^5u_ia_{ik}u_kb_{kj}u_j\right).$$

\vspace{5mm}

\subsection*{Acknowledgments}
J.-W. He was supported by ZJNSF (No. LY19A010011), NSFC (No. 11971141). Y. Ye was supported by NSFC (No. 11971449).

\vspace{5mm}

%\bibliography{}

\end{document}